\documentclass[12pt]{amsart}

\usepackage{dsfont}
\usepackage{bm}
\usepackage{comment}
\usepackage{bbm}
\usepackage{fourier}
\usepackage{comment}
\usepackage{mathtools}
\usepackage[]{hyperref}

\newtheorem{prop}[equation]{Proposition}
\newtheorem{lemma}[equation]{Lemma}

\newtheorem{mainthm}{Theorem}

\theoremstyle{definition}

\newtheorem{rem}[equation]{Remark}
\theoremstyle{remark}

\renewcommand{\theequation}{\thesection.\ifnum\value{subsection}=0 1\else \arabic{subsection}\fi.\arabic{equation}}

\DeclareMathOperator{\supp}{supp} % The support.

\usepackage[usenames]{color}
\usepackage[table,xcdraw]{xcolor}
\usepackage{tikz}
\usepackage[framemethod=tikz]{mdframed}
\definecolor{aliceblue}{rgb}{0.92, 0.93, 1.0}

\newcounter{change}

\usepackage[colorinlistoftodos, textwidth=4cm]{todonotes}

\begin{document}

\title{Discontinuous observables as an obstruction for small essential spectral radius}

\author{Oliver Butterley}
\address{(Oliver Butterley) Department of Mathematics, University of Rome Tor Vergata, Via della Ricerca Scientifica 1, 00133 Roma, Italy}
\email{butterley@mat.uniroma2.it}
\urladdr{https://www.mat.uniroma2.it/butterley/}

\author{Daniel Smania}
\address{(Daniel Smania) Departamento de Matemática, Instituto de Ciências Matemáticas e de Computação (ICMC), Universidade de São Paulo (USP), Avenida Trabalhador São-carlense, 400, São Carlos, SP, CEP 13566-590, Brazil}
\email{smania@icmc.usp.br}
\urladdr{https://sites.icmc.usp.br/smania/}

\begin{abstract}
We show that for a very wide class of Banach spaces of functions on $[0,1]$ there are intrinsic lower bounds for the essential spectral radius of the transfer operator associated to piecewise smooth expanding maps.
The class of Banach spaces studied includes any reasonable space which permits discontinuities.
\end{abstract}

\subjclass[2020]{37A05, 37A25, 	37C30 }

\keywords{transfer operator, essential spectrum, spectral gap, piecewise expanding map, discontinuities}

\maketitle

%\setcounter{tocdepth}{1}
%\tableofcontents

%%%%%%%%%%%%%%%%%%%%%%%%%%%%%%%%%%%%%%%%%%%%%%%%%%%%%%%%%%%%%%%%%%%%%%
\section{Introduction}
%%%%%%%%%%%%%%%%%%%%%%%%%%%%%%%%%%%%%%%%%%%%%%%%%%%%%%%%%%%%%%%%%%%%%%

Given a chaotic system with some degree of hyperbolicity it is natural to investigate the statistical properties of the system.
This involves finding and studying relevant invariant measures, proving CLT, LLT, large deviations, estimating the decay of correlations, studying zeta functions, etc. 
Study of the transfer operators associated to dynamical systems is a very convenient and powerful way to investigate statistical properties.
Such an idea goes back at least to the use of the Koopman operator by von Neumann to prove the mean ergodic theorem.
Subsequently Sinai and others in the Russian school developed theory for the Koopman operator acting on \(L^2\) and the connection with ergodicity and mixing (see e.g., \cite{CFS82}).
Soon it was realised that it was useful to study the adjoint of the Koopman operator and this object became known as the transfer operator.
Amongst this period of development is the work of Lasota \& Yorke~\cite{LY73}, Ruelle~\cite{Ruelle89}, Keller \cite{Keller84,Keller89}, Baladi \& Keller~\cite{BK90}, Keller \& Liverani~\cite{KL99}.
See the books by Baladi~\cite{baladib,Baladi18} for a more complete history and the notes of Liverani~\cite{Liverani04} for an overview.
In many cases the use of transfer operators was done by reducing the system to symbolic dynamics and then using standard function spaces (see Bowen \cite{Bowen08} and Parry and Pollicott \cite{pp}).
On the other hand, in some cases one could choose dynamically sensible choices of Banach space on which to study the transfer operator without the need of coding and potential loss of information.
In the case of hyperbolic systems, as opposed to expanding systems, this led to the need of anisotropic Banach spaces of distributions to match with the distinctly different behaviour in different directions (see Blank, Keller, and  Liverani\cite{BKL02}, Gouëzel and Liverani  \cite{GL06} and Baladi and Tsujii \cite{BT07}).
At present there are many different Banach spaces available for studying many different piecewise expanding maps~\cite{besti} and the approach has shown many great successes.

The spectrum of the transfer operator will consist of essential spectrum and a set of isolated eigenvalues, The  spectrum depends on the choice of Banach space.
In order to be useful for studying the system the essential spectrum radius needs to be smaller than the spectral radius. 
For some systems, those with high degree of regularity, there exists a choice of Banach space so that the essential spectral radius is arbitrarily small. 
For example this has been shown for smooth expanding maps in any dimension (see Collet and Isola \cite{CI91},  Gundlach and  Latushkin \cite{GL03}) and pseudo Anosov diffeomorphisms by Faure, Gouëzel and  Lanneau \cite{FGL19}. 
This corresponds to being able to give a precise description of decay of correlations  in terms of resonances~\cite[Definition 1.1.]{FGL19}.
Choosing a Banach space which is larger than required can lead to the essential spectrum being large.
For instance, studying the transfer operator acting on \(L^p\) with \(1 \leq p < \infty\) an approximate eigenfunction argument can be used to show that the spectrum of the transfer operator is the entire disc, i.e., the essential spectrum is equal to the spectral radius~\cite[Footnote 8]{BCJ23}.
One direction of interest is to identify the isolated points of spectrum (see \cite{KR04,BKL22,Liverani01} and references within).
Indeed the isolated points of spectrum can be shown to be essentially independent on the choice of Banach space~\cite{BT08}.

The focus of this present work is in the other direction, on the essential spectrum and understanding if it can be reduced in a given situation.
Showing that the essential spectral radius is small is intimately connected with showing a large meromorphic extension of the zeta function (see, e.g., \cite{FLP94,Baladi18}). 

Given the huge availability of different Banach spaces and the unlimited creative possibility it is natural to ask if, in a given situation, there is a chance of finding a better Banach space in order to reduce the essential spectral radius. 
As hinted above, if the Banach space is too large the essential spectral radius is large.
On the other hand the Banach space must be sufficiently large in order to be useful, typically one would want it to include at least all smooth observables.
It was shown by B., Canestrari \& Jain~\cite{BCJ23} that, for the case of smooth interval maps with discontinuities, the essential spectrum is large for a very large class of observables (see \cite{BCC23} for a higher dimensional extension). 
The only requirement on the Banach space (apart from containing \(\mathcal{C}^\infty\) and being invariant under the dynamics) was that it had to be continuously embedded in \(L^\infty\).
This, as already noted in that work, is unfortunate since there are many spaces, for example Besov spaces (see Arbieto and S. \cite{as2020transfer}, S. \cite{besti}  and Nakano and Sakamoto \cite{ns})  and Sobolev spaces (see Thomine \cite{Thomine11})  which seem like a reasonable choices but are not embedded in \(L^\infty\).
Discontinuities are natural in physical systems (e.g. see Chernov and Markarian \cite{CM06}) but, additionally, it can be argued that some physically relevant observables are unbounded and such should therefore be permitted in the analysis.
Rectifying this gap is a major motivation in the present work.

Here, we will show that {\it for a broad class of Banach spaces of observables, satisfying fairly minimal conditions — in particular, allowing simple discontinuities in the observables — the action of the transfer operator has a large essential spectral radius.} These classes of Banach spaces are sufficiently general to include unbounded observables.

In Theorem~\ref{main} we give conditions which imply that Besov space is embedded in a given Banach space and also derive a lower estimate for the essential spectral radius.
Theorem~\ref{BB} applies to linear expanding maps and again gives a lower bound for the essential spectral radius but with weaker assumptions on the Banach space.
On the other hand, Theorem~\ref{NEW} requires an assumption of the topological pressure in order to give a lower bound (an assumption that holds for $C^\infty$ maps).
In Theorem~\ref{thm2} we apply the ideas to the case where the norm is \emph{natural} (defined in Section~\ref{sec:natural}) and we obtain a lower estimate on the essential spectral radius.

Let $I=[0,1]$ and 
$$T\colon \cup_i I_i \rightarrow I$$ where $\{ I_i\}_i$  is a finite partition of $I$ by open intervals, and $T\colon I_i \rightarrow T(I^i)$ extends to a $C^1$ diffeomorphism on $\overline{I}_i$.  Then the transfer operator $L$ with respect to the is a bounded operator acting on $L^1(I)$. Denote by $r(L,L^1(I))$ its spectral radius.    

\subsection{Growth of derivatives and topological entropy}

Let $$\mathcal{P}^k=\{I^k_i\}_i$$ be the intervals of monotonicity of $T^k$. Let $$\theta^k_i =\sup_{x\in I^k_i} \frac{1}{|Df^k(x)|}.$$
Notice that 
$\Theta^k(\beta) = \sum_i (\theta_i^k)^\beta$
is sub-multiplicative, $\Theta^{k+j}(\beta)\leq \Theta^{k}(\beta)\Theta^{j}(\beta)$,
so we can define 
$$\Theta^\infty(\beta) = \lim_k (\Theta^{k}(\beta))^{1/k}.$$ 

\begin{rem} An easy upper bound is 
$$\Theta^\infty(\beta)\leq \#\mathcal{P}^1 \lim_k \Big( \sup_{x\in I} \frac{1}{|Df^k(x)|^\beta}\Big)^{1/k}.$$
If $T$ is continuous one can be more precise  

$$\Theta^\infty(\beta)\leq e^{h_{top}(T)} \lim_k \Big(\sup_{x\in I} \frac{1}{|Df^k(x)|^\beta}\Big)^{1/k}.$$
Here $h_{top}(T)$ is the topological entropy of $T$.   If $T$ is an $C^{1+\beta}$  expanding map on the circle we observe that 
$$\lim_k \Big(\sup_{x\in I} \frac{1}{|Df^k(x)|^\beta}\Big)^{1/k}=e^{-\beta M},$$
where 
$$M= \inf_{\mu inv. prob. T} \int \ln |Df| \ d\mu.$$
Indeed in this case $$\Theta^\infty(\beta)=e^{P_{top}(-\beta \ln |Df|)},$$
where $P_{top}(\phi)$ denotes the topological pressure of $\phi$.
\end{rem}

\section{Main results} 
 \subsection{Lower bound for essential spectrum radius} Discontinuities of a  dynamical system is a serious obstruction for  small essential spectrum radius of the transfer operator $L$ acting on Banach spaces of bounded functions (see B., Canestrari and Jain~\cite{BCJ23}). 
% note: 
% In BCJ22 there was the unfortunate limitation of assumption (3) of Theorem 1, here we explore this limitation.
% We show for these Besov spaces the EssSpec is large, even if there are no discontinuities. 
Here we are going to see that if the transfer operator acts on a space of functions $B$ that contains the simplest of discontinuous functions, and satisfies a basic norm estimate, then $B$ is indeed quite large and its essential spectrum cannot be small. 

Given $s\in (0,1)$, let $B^s_{1,1}$ be the classical space of Besov function on the interval $I$. Denote by $|A|$ the Lebesgue measure of the set $A$.
 
\begin{mainthm} \label{main} Suppose that \begin{itemize}
\item[-] The map  $T\colon I_i \rightarrow T(I^i)$ extends to a $C^{1+\beta}$ diffeomorphism on $\overline{I}_i$, for every $i$, with $\beta > 0$. 
\item[-]  The Lebesgue measure $m$ on  $I$ is $T$-invariant.
\item[-] The transfer operator $L$ associated with $T$, with respect to the Lebesgue measure $m$, preserves a Banach space of functions $B$ that is continuously embedded in $L^1(I)$ and the operator $L \colon B \to B$ is bounded.
\item[-] There are $C\geq 0$ and $s \in (0,1)$ such that for every interval $J\subset I$ we have that $1_J\in B$ and 
\begin{equation} \label{estint}||1_J||_B\leq C |J|^{1-s}
\end{equation}
\end{itemize} 
Then the Besov space $B^s_{1,1}$ is continuously embedded in $B$.  If furthermore   $s <  \beta$ we have
$$r_{ess}(T,B)\geq 1/\Theta^\infty(1-s)$$
and every $\lambda \in \mathbb{C}$ with $|\lambda|< 1/\Theta^\infty(1-s)$ is an eigenvalue of $L$ on $B$ with an infinite dimensional eigenspace.
\end{mainthm} 

In the piecewise linear case, we can weaken the assumption and still obtain a lower bound for the essential spectrum radius.

\begin{mainthm}\label{BB}
Suppose that 
\begin{itemize}
\item[-]  The Lebesgue measure $m$ on  $I$ is $T$-invariant.
\item[-]  $T$ is linear on each branch and it has $k$ branches. 
\item[-]The transfer operator $L$ associated with $T$, with respect to the Lebesgue measure $m$, preserves a Banach space of functions $B$ that is continuously embedded in $L^1(I)$ and the operator $L \colon B \to B$ is bounded. 
\item[-] There is $C\geq 0$ such that for every interval $J\subset I$ we have that $1_J\in B$ and 
\begin{equation} \label{estint2}||1_J||_B\leq C.
\end{equation}
\end{itemize} 
Then 
$$
r_{ess}(L,B)\geq 1/k
$$
and every $\lambda \in \mathbb{C}$ which satisfies $|\lambda|< 1/k$ is an eigenvalue of $L$ on $B$ with an infinite dimensional eigenspace.
\end{mainthm}

\begin{mainthm}\label{NEW}
Suppose that 
\begin{itemize}
\item[-]  The Lebesgue measure $m$ on  $I$ is $T$-invariant.
\item[-]  $T$ is piecewise $C^{r+1}$ and expanding, for some $r> 0$,  Markovian with $k$ branches, all of them onto,  and
\begin{equation}\label{hipp} P_{top}(-(r+1)\log |DT|)< \frac{1}{k}. \end{equation} 
\item[-]The transfer operator $L$ associated with $T$, with respect to the Lebesgue measure $m$, preserves a Banach space of functions $B$ that is continuously embedded in $L^1(I)$ and the operator $L \colon B \to B$ is bounded. 
\item[-] There is $C\geq 0$ such that for every interval $J\subset I$ we have that $1_J\in B$ and 
\begin{equation} \label{estint3}||1_J||_B\leq C.
\end{equation}
\end{itemize} 
Then 
$$r_{ess}(L,B)\geq \frac{1}{k}.$$
%Moreover if $B$  is continuously embedded in $C^{r}(I)$ then for  every $z\in \mathbb{C}$ such that  \begin{equation*}  \exp(P_{top}(-(r+1)\log |DT|))  < |z|<  \frac{1}{k}.\end{equation*} 
%there is $\delta \in \mathbb{C}$, with $\delta^n=1$, for some $n \in \mathbb{N}^\star$, such that $\delta z$ is an eigenvalue of $L$ in $B$.
\end{mainthm}
\begin{rem} Condition (\ref{hipp}) is satisfied 
\begin{itemize}  
\item[-] for every $T$ that is piecewise expanding  $C^{r+1}$ and    Markovian with $k$ branches, and
$$\sup_{x\in I} \frac{1}{|DT(x)|^r} < \frac{1}{k}.$$
\item[-] for every $T$  that is piecewise expanding $C^{\infty}$ and    Markovian with $k$ branches, provided $r$ is large enough.
\end{itemize} 
\end{rem} 

\begin{rem} Every $C^r$ Markovian  expanding map, $r > 1$,  with only full branches, is conjugate by a $C^r$ conjugacy with a $C^r$ Markovian  expanding map that preserves the Lebesgue measure, so this assumption is not strong. 
\end{rem} 

\begin{rem} Note that if $L\phi =\lambda \phi$ with $\phi \in B\setminus \{0\}$ and $\psi(x) = \overline{\phi}(x)/|\phi(x)|$ if $\phi(x)\neq 0$, and $\psi(x)=0$ otherwise, then
$$\int \psi\circ T^n \phi  \ dm = \lambda^n \int |\phi| \ dm, $$
so $\lambda$-eigenvectors are obstructions for decay of correlations faster than $|\lambda|$. 
\end{rem}

\subsection{Natural spaces of functions} 
\label{sec:natural}
Conditions (\ref{estint}) and (\ref{estint2}) appear challenging to verify. However, most Banach function spaces on $\mathbb{R}$ discussed in the literature exhibit good behaviour with respect to translation and scaling, which will significantly aid our analysis. \\

\noindent {\bf Almost Homogeneity and invariance by translations.} 
We will say that a pseudo-norm $n(\cdot)$ on a   space of functions on the interval $I$  is {\it purely natural} if there is $t \in \mathbb{R}$ and $C> 0$  such that  if $u \colon \mathbb{R} \rightarrow \mathbb{R}$ is  an invertible affine transformation and $\phi\colon I \rightarrow \mathbb{C}$ satisfies 
$$u^{-1}(\supp \ \phi)\subset I$$
then $\phi \circ u \in  B$ and 
$$\frac{1}{C} |u'|^t 
 n(\phi)\leq n(\phi\circ u)\leq C |u'|^t n(\phi).$$

 \ \\
The parameter $t$ will be called the degree of homogeneity of $u$. 

A pseudo-norm $n$ is called {\it natural} it is a finite sum of purely natural pseudo-norms, that is, there are purely natural pseudo-norm $n_i$, with $i\leq j$, and degree of homogeneity $t_i$ such that 
\begin{equation}\label{pn}  n(\phi)=\sum_{i\leq j} u_i(\phi).\end{equation} 

\begin{rem} Many norms and pseudo-norms  of spaces of function of an interval are natural. The sup norm, the $L^p$ norms, the $p$-bounded variation pseudo-norm, the H\"older norm, $C^k$ norms, Sobolev norms and Besov norms. 
\end{rem}

A nice application of the previous results is 

\begin{mainthm}\label{thm2} 
Suppose that 
\begin{itemize}
\item[-]  $T$ is piecewise $C^{r+1}$ and expanding acting on the interval $I=[0,1]$, for some $r> 0$,  Markovian with $k$ branches, all of them onto,  and
\begin{equation}\label{hipp2} P_{top}(-(r+1)\log |DT|)< \frac{1}{k}. \end{equation} 
\item[-]  The Lebesgue measure $m$ on  $I$ is $T$-invariant.
\item[-]  $B$ is  a Banach space of functions continuously embedded in $L^1(I)$ whose  norm $||\cdot||_B$ is natural. 
 \item[-] The transfer operator $L$ of $T$ with respect to the Lebesgue measure $m$ keeps a Banach space of functions $B$ invariant, and $L\colon B \rightarrow B$ is a bounded operator with spectral gap.
\item[-] For every interval $J\subset I$ we have that $1_J\in B$.
\end{itemize} 
Then one of the following cases occurs 
\begin{itemize}
\item[I.] There is $C > 0$ such that 
$$\frac{1}{C}\leq   ||1_P||_B\leq C  $$
for every interval $P\subset I$, and 
$$r_{ess}(L,B)\geq 1/k.$$
\item[II.]
There is $s\in (0,1)$ and $C> 0$ such that   
$$\frac{1}{C}|P|^{1-s}\leq ||1_P||_B\leq C |P|^{1-s}$$
for every interval $P \subset I$, the Besov space $B^s_{1,1}$ is continuously embedded in $B$ and 
$$r_{ess}(L,B)\geq \frac{1}{\Theta^\infty(1-s)}.$$
\end{itemize}
\end{mainthm} 

\begin{rem} 
Both cases occur. If $B$ represents the space of bounded variation functions on $I$ and $T(x) = 2x \mod 1$ on $I = [0,1]$, then we are in Case I. On the other hand, if we consider $B^s_{1,1}$ with $s \in (0,1)$, we are in Case II (see Nakano and Sakamoto \cite{ns} and S. \cite{besti}). Furthermore, note that if $B \subset L^\infty$, we must be in case I, as $B^s_{1,1}$ includes unbounded functions.
\end{rem} 

\section{Preliminaries}

\subsection{$p$-bounded variation} 
Let $J=[a,b]\subset I$. The $p$-variation of a function $\psi\colon I \rightarrow \mathbb{C}$  on the interior of $J$ is
$$v_p(\psi,J)= \sup (\sum_{i=0}^n  |\phi(x_{i+1})-\phi(x_i)|^p)^{1/p}  ,$$
where the sup runs over all possible finite increasing sequences
$$a< x_0 < x_1 < \cdots < x_{n-1} < x_n < b.$$
\begin{lemma} \label{prop} The pseudo-norm $v_p$ has the following properties properties
\begin{itemize}
\item $v_p$ is invariant with respect to continuous change of coordinates: if $u\colon P\rightarrow J$ is a homeomorphism then 
$$v_p(\phi,J)= v_p(\phi\circ u ,P).$$
\item  if  $J_1$ and $J_2$ are intervals  such that $J_1\cap J_2$ is just a point then 
$$v_p^p(\psi,J_1)+ v_p^p(\psi,J_2)\leq v_p^p(\psi,J_1\cup J_2).$$
\end{itemize}
\end{lemma}

\subsection{ Besov space $B^s_{1,1}(I)$} Given $s \in (0,1)$, consider the space of all  functions $\psi\in L^1(I)$ that can be written as 

\begin{equation}\label{rep} \psi =\sum_{n=0}^\infty c_n |Q_n|^{s-1}1_{Q_n},\end{equation} 
where this series converges in $L^1(I)$ and $Q_n$ are subintervals of $I$. If we endow $B^s_{1,1}$ with the norm 
$$|\phi|_{B^s_{1,1}(I)}=\inf \sum_n |c_n|, $$
where the infimum runs over all possible representations, $B^s_{1,1}$ is  a Banach space.   

Those spaces were introduced by de Souza \cite{souza-atomicdec}. There are indeed many way to describe this space (see de Souza \cite{souza2}). In particular it coincides with  the classical Besov spaces $B^s_{1,1}(I)$. The proof of the following proposition is quite simple.

\begin{prop} Let $B$ be as in the Theorem \ref{main}. There is a continuous embedding 
of $B^s_{1,1}$ in $B$, that is, $B^s_{1,1}\subset B$ and there is $C$ such that 
$$||\phi||_{B}\leq C||\phi||_{B^s_{1,1}}.$$
\end{prop} 

\begin{prop}\label{esti}
Given $1/p > s$ there is $C\geq 0$ such that the following holds. For   every function $\psi\colon I \rightarrow \mathbb{C}$ that vanish outside a closed  interval $J$ and $v_p(\psi,J)< \infty$ we have that $\psi\in  B^s_{1,1}(I)$ and
$$|\psi- m(\psi,J)1_J|_{B^s_{1,1}(I)}\leq C |J|^{1-s} v_p(\psi,J).$$
Here
$$m(\psi,J)= \frac{1}{|J|} \int_J \psi  \ dm.$$
Note that $C$ does not depend on $J$. 
\end{prop} 
 \begin{proof} We do a argument  similar to the proof of Proposition 16.3 in S. \cite{smania_2022-PDE}. Let $\mathcal{D}^k$ be the partition of $J$ by intervals of length $|J|/2^{-k}$. Then
 $$\psi = \lim_k \psi_k,$$
  in $L^1(m)$, where $$
 \psi_k = \sum_{P\in \mathcal{D}^k}  m(\psi,P) 1_P.$$
Note that 
$$\psi_0= \frac{1}{|J|} \int_J \psi \ dm.$$
We claim that sequence  converges in $B^s_{1,1}$.  Indeed
\begin{align*} &\psi_{k+1}-\psi_k \\
=&\sum_{P\in \mathcal{D}^{k+1}}  m(\psi,P) 1_P- \sum_{Q\in \mathcal{D}^{k}} m(\psi,Q) 1_Q\\
=&\sum_{Q\in \mathcal{D}^k} \sum_{\substack{P\in \mathcal{D}^{k+1}\\ P\subset Q}} (m(\psi,P)-  m(\psi,Q)) 1_P \end{align*} 
and
$$|m(\psi,P)-  m(\psi,Q)|\leq 2 v_p(\psi,Q), $$
so
\begin{align*} &\sum_{k\geq 0} |\psi_{k+1}-\psi_k|_{B^s_{1,1}(I)} \\
\leq & \sum_{k\geq 0} \sum_{Q\in \mathcal{D}^k} |Q|^{1-s}  4 v_p(Q)\\
\leq & 4 \sum_{k\geq 0}   \big( \sum_{Q\in \mathcal{D}^k} |Q|^{\frac{(1-s)p}{p-1}} \big)^{\frac{p-1}{p}}    \big( \sum_{Q\in \mathcal{D}^k} v_p^p(Q)\big)^{1/p}\\
& \leq 4 \sum_{k\geq 0}   \big( \sum_{Q\in \mathcal{D}^k} |Q|^{1+ \frac{1-sp}{p-1} }\big)^{\frac{p-1}{p}}    \big( \sum_{Q\in \mathcal{D}^k} v_p^p(Q)\big)^{1/p}\\
&  \leq  C |J|^{1-s} v_p(Q).
\end{align*} 
 \end{proof} 
\section{Proof of the main results} 
 \begin{proof}[Proof of Theorem \ref{main}]  
Let $\psi \in L^\infty(I)$, with $\psi \neq 0$ and $L\psi=0$. 
Note that 
$$|\psi\circ T^\ell|_{L^p(I)} \leq |\psi|_{L^\infty}$$ 
for every $p\in [1,\infty]$. Given $z \in \mathbb{C}$ with $|z|< 1$, define 

\begin{equation}
        \label{eq:def-hz}
        h_z = \sum_{\ell=0}^{\infty} z^{\ell} \psi \circ T^{\ell} \in L^p(I), \ with \ p\in [1,\infty].
    \end{equation} 
    Let 
    $$\mathcal{F}_\ell= \{\phi\circ T^\ell\colon \ \phi \in L^\infty(I) \ and \ L\phi=0 \}.$$
Then $\mathcal{F}_\ell$ are mutually orthogonal on $L^2(I)$ (see de Lima and S. \cite{amanda}). In particular  $h_z\neq 0$ since 
$$\{ \psi\circ T^\ell\}_\ell$$
is an orthogonal set in $L^2(I)$. Since $m$ is $T$-invariant we have that $L(\phi\circ T)=\phi$ for every $\phi \in L^\infty(I)$, and an  easy calculation shows that 
$$ Lh_z=zh_z.$$
That is a classical construction of eigenvalues of $L$ (see Collet and Isola\cite{CI91}) and indeed it shows that every $|z|< 1$ is an eigenvalue of $L$ with an infinite dimensional eigenspace on $L^p(I)$, for every   $p\in [1,\infty]$. 
% note: the difference with CI91 is that they used smooth functions in the kernel of the trans op whereas here we use a discontinuous function in the kernel of the trans op.

 Since $T$ is piecewise $C^{1+\beta}$ we  can choose disjoint closed intervals $I_1, I_2 \subset I$ and $\beta$-H\"older functions $\phi_i\colon J_i \rightarrow \mathbb{R}$, $i=1,2$,  such that  
\begin{align*} &\psi= \phi_1 1_{J_1} - \phi_2 1_{J_ 2} \in L^\infty(I),\\ 
&v_{1/\beta}(\phi_i,J_i)< \infty,\end{align*} 
and  $L\psi=0$, with $\psi \neq 0$. Indeed note that the set of all possible such $\psi$ is infinite dimensional. We have that 
$$(\phi_j1_{I_j})\circ T^k = \sum_{i} \phi_j\circ T^k 1_{Q_i^k},$$
where $Q_i^k \subset I^k_i$ is an interval mapped by $T^k$ monotonically to a sub-interval of $I_j$. In particular 
\begin{align*} 
&|Q_i^k|\leq \theta^k_i,\\
&\sup_ i \sup_ k v_{1/\beta}(\phi_i\circ T^k,Q^k_i)< \infty,
\end{align*} 
In the last estimate we used Lemma \ref{prop}.  Since $\beta > s$,  due Proposition \ref{esti}  we have
$$|\psi\circ T^k|_{B^s_{1,1}}\leq C \sum_i (\theta_i^k)^{1-s} =C \Theta^k(1-s). $$
So if $|z|< 1/\Theta^\infty(1-s)$ we have that $h_z \in B^s_{1,1}\subset B$. Consequently 
$$r_{ess}(T,B)\geq 1/\Theta^\infty(1-s)$$
and
$$r_{ess}(T,B^s_{1,1})\geq 1/\Theta^\infty(1-s).$$
 \end{proof} 

\begin{proof}[Proof of Theorem \ref{BB}] We use the same notation as in the proof of Theorem \ref{main}. Since $T$ is linear on each branch we can  choose disjoint intervals $I_1$ and $I_2$ in $I$ and {\it positive constants}  $c_1$ and $c_2$ such that $\psi = c_1 1_{I_1}- c_21_{I_2}\in L^\infty(I)$, $L\psi=0$ and $\psi \neq 0$. In this case $\psi\circ T^k$ is a linear combinations of characteristic functions of intervals. So for $|z|< 1/k$ we have 
    \begin{equation}
        \label{hz}
        h_z = \sum_{\ell=0}^{\infty} z^{\ell} \psi\circ T^{\ell}.
    \end{equation}\
indeed converges in $B$ since 
$$|\psi\circ T^\ell|_B\leq C k^\ell.$$
So $r_{ess}(L,B)\geq 1/k$.
\end{proof}

\begin{proof}[Proof of Theorem \ref{NEW}] 
Suppose, for the sake of contradiction, that
\begin{equation}\label{pop23} r_{ess}(L,B) < 1/k.\end{equation}

Since $T$ is a piecewise $C^{r+1}$ markovian map, by Collet and  Isola \cite{CI91} (see also Baladi \cite{baladib})  we have that 
 $$ r_{ess}(L,C^{r}(I))=\exp(P_{top}(-(r+1)\log |DT|)) <  \frac{1}{k}.$$
 Let $z\in \mathbb{C}$ be  such that 
\begin{equation}\label{condz}  \max \{ \exp(P_{top}(-(r+1)\log |DT|)), r_{ess}(L,B)\}   < |z|<  \frac{1}{k}.\end{equation}
 
%\begin{equation}\label{condz}  \max \{ r_{ess}(L,B), \exp(P_{top}(-(r+1)\log |DT|)) \} < |z|<  \frac{1}{k}.\end{equation} 
 In particular  there is a finite dimensional subspace $E\subset C^{r}(I)$ and a closed subspace $F\subset C^{r}(I)$ such that 
 \begin{itemize}
 \item[i.] $E$ and $F$ are $L^n$-invariant for every $n > 0$.
 \item[ii.] $r(L^n,F)< |z|^n$ for every $n > 0$.
 \item[iii.] $B= E\oplus F$
 \end{itemize} 
 and  there is a finite dimensional subspace $\hat{E}\subset B$ and a closed subspace $\hat{F}\subset B$ such that 
 \begin{itemize}
 \item[iv.] $\hat{E}$ and $\hat{F}$ are $L^n$-invariant for every $n > 0$.
 \item[v.] $r(L^n,\hat{F})< |z|^n$ for every $n > 0$.
 \item[vi.] $B= \hat{E}\oplus \hat{F}$
 \end{itemize} 

 Let $\mathcal{P}^n$ be the partition of $I$ by the intervals of monotonicity of $T^n$ and $\mathcal{F}^n\subset B$ be the space of functions that are constant of each element of $\mathcal{P}^n$. Note that if   $L^{n_0}\mathcal{F}^{n_0} \subset C^{r}(I).$ \ \\
 
 \noindent {\it Claim A. If  $n_0 > \dim E + \dim \hat{E} +1$ then there is $\psi \in\ \mathcal{F}^{n_0}$ that is not constant, $\psi \in \hat{F}$, and $L^{n_0}\psi\in F$.}
  \  \\ 
\noindent  Let $\pi_E,\pi_F$ be the  orthogonal projections associated with the decomposition $C^r(I)=E\oplus F$, and $\pi_{\hat{E}},\pi_{\hat{F}}$ be the  orthogonal projections associated with the decomposition $B=\hat{E}\oplus \hat{F}$.  Define
$$H\colon \mathcal{F}^{n_0}\rightarrow E \oplus \hat{E}$$ 
as $H(\psi)= ( \pi_E(L^{n_0}\psi), \pi_{\hat{E}}(\psi))$. If $n_0 > \dim E+\dim \hat{E}+ 1$ we have that $\dim Ker \ H\geq 2$, so it contains a non-constant function. This finishes the proof of the claim.  Choose $n_0$ and $\psi$ as in Claim A. \ \\ 

  \noindent {\it Claim B.   For every $n> 0$ we have that 
   \begin{equation}
        \label{hzf}
        h_{z,n} = \sum_{\ell=0}^{\infty} z^{\ell n } \psi\circ T^{\ell n}
    \end{equation}
    converges in $B$ and 
  $$L^nh_{z,n} = z^n h_{z,n} + L^n\psi.$$
  Moreover  for large $n$ we have that the image of $h_{z,n}$   is a Cantor set (up to a countable set). 
  
  } \ \\ 
    
\noindent Note that   $\psi\circ T^\ell$ is a linear combination of characteristic functions of intervals and  we have 
    \begin{equation}
        \label{hz2}
        h_{z,n} = \sum_{\ell=0}^{\infty} z^{\ell n } \psi\circ T^{\ell n}.
    \end{equation}\
indeed converges in $B$ since by (\ref{estint3}) 
$$|\psi\circ T^\ell|_B\leq C k^\ell.$$
 One can easily check that 
\begin{equation}\label{cohomo}   -z^{-n} \psi = h_{z,n}\circ T^n -  z^{-n } h_{z,n},\end{equation} 
and $h_{z,n}$ is the unique bounded function that is a  solution of such cohomological equation. Let 
$$\mathcal{I}= \{y \in \mathbb{C}\colon \ \psi= y \text{  in   some } P\in \mathcal{P}^{n_0} \}.$$
Since $\psi$ is not constant we have that $\#\mathcal{I}\geq 2$.  For every $q \in \mathcal{I}$ define  the affine map
$$\phi_{q,n} \colon \mathbb{C} \rightarrow \mathbb{C}$$
given by 
$$\phi_{q,n}(u)= z^{-n}u -z^{-n}q.$$
One can rewrite (\ref{cohomo}) as 
$$\phi_{\psi(x),n}  \circ h_{z,n}(x) =h_{z,n}\circ T^n(x).  $$

Observe that the unique fixed point of $\phi_{q,n}$ is 
$$x_{q,n}=\frac{q}{1-z^n},$$
and $\lim_n x_{q,n} =q$. %Let 
%$$\epsilon = \min_{q_1,q_2\in \mathcal{I}, q_1\neq q_2} |q_1-q_2| > 0.$$
Let $R= 2 diam \ \mathcal{I}$ and choose $q_0\in \mathcal{I}$. If   $n> n_0$ is  large enough 
\begin{itemize} 
\item[-]  We have $$\phi_{q,n}^{-1}(B(q_0,R))\subset B(q_0,R),$$
for every $q\in \mathcal{I}$,
\item[-] $\phi_{q_1,n}^{-1}(B(q_0,R))$ and $\phi_{q_2,n}^{-1}(B(q_0,R))$ are disjoint for every $q_1,q_2\in \mathcal{I}$ with $q_1\neq q_2$. 
\end{itemize} 
In particular the map
$$G_n\colon \bigcup_{q\in \mathcal{I}}  \phi_{q,n}^{-1}(B(q_0,R)) \rightarrow B(q_0,R)$$
defined by $G_n(x)= \phi_{q,n}(x)$ for $x\in\phi_{q,n}^{-1}(B(q_0,R))$,  is a conformal expanding map and its maximal invariant set 
$$\Omega_n =\bigcap_{\ell >0}  G^{-\ell}_n B(q_0,R)$$
is a Cantor set.  Let 
$$\hat{h}_{z,n}\colon I \rightarrow \mathbb{C}$$
be the bounded function defined by 
$$\hat{h}_{z,n}(x)= \lim_k   \phi_{\psi(x),n}^{-1}\circ \phi_{\psi(T^nx),n}^{-1}\circ \phi_{\psi(T^{2n}x),n}^{-1}\cdots \circ \phi_{\psi(T^{kn} x),n}^{-1}(q_0)\in \Omega_n\cap \phi^{-1}_{\psi(x),n}(B(q_0,R))$$
Since $T^n$ has only full branches  the image of $\hat{h}_{z,n}$ is the Cantor set $\Omega_n$. 
It is easy to see that 
$$G_n  \circ \hat{h}_{z,n}(x) =\hat{h}_{z,n}\circ T^n(x),  $$
that is
$$\phi_{\psi(x),n}  \circ \hat{h}_{z,n}(x) =\hat{h}_{z,n}\circ T^n(x),  $$
so 
$$ -z^{-n} \psi = \hat{h}_{z,n}\circ T^n -  z^{-n } \hat{h}_{z,n}.$$
and consequently $\hat{h}_{z,n}=h_{z,n}$. 
 \ \\ 

\noindent {\it Claim C.  For $n\geq n_0$  we have that 
$$w_{z,n} = \sum_{\ell=1}^\infty z^{-\ell n} L^{\ell n} \psi$$
converges in $B$ and $C^r(I)$, and moreover   \begin{equation} \label{popp}  L^nw_{z,n}=z^n w_{z,n} - L^n\psi.\end{equation} 
}
\ \\
Since  $L^{n_0}\psi\in F\cap \hat{F}$ the above series converges in $B$ and  $C^r(I)$.  It is easy to verify (\ref{popp}).\\

\noindent {\it Claim D.  For large $n$ we  have that $h_{z,n}+w_{z,n} \neq 0$ and 
\begin{equation}\label{rrr} L^n(h_{z,n}+w_{z,n})=z^n(h_{z,n}+w_{z,n}).\end{equation} }

\noindent The equality (\ref{rrr}) is obvious.  Note that due Claim B. we have  that for large $n$ the image of $h_{z,n}$ is a Cantor set (up to a countable set), and the image of $-w_{z,n}$ is a  (perhaps empty) interval (up to a finite subset). So $h_{z,n}+w_{z,n} \neq 0$ and $z^n$ is an eigenvalue of $L^n$. In particular there is $\delta \in \mathbb{C}$, with $\delta^n=1$, such that $\delta z$ is an eigenvalue of $L$. 

Since $z$ can be an arbitrary complex number satisfying (\ref{condz}), Claim D. implies that $r_{ess}(L^n,B)\geq 1/k^n$ and  $r_{ess}(L,B)\geq 1/k$.
%If $C^r(I)$ is continuously embedded in $B$ then the above construction of a $\delta z$-eigenvector  can be done for every $z$ satisfying
% $$\exp(P_{top}(-(r+1)\log |DT|))   < |z|<  \frac{1}{k}.$$
\end{proof}

 \begin{proof}[Proof of Theorem \ref{thm2}]

  Since  $n(\phi)=||\phi||_B$ is natural, it can be written as in (\ref{pn}), where $u_i$ there are purely natural pseudo-norm $n_i$, with $i\leq j$, and degree of homogeneity $t_i$. %Consequently

%$$\frac{1}{C} 2^{k(1+t)}
 %||a_I||_B \leq ||a_P||_B\leq C 2^{k(1+t)} ||a_I||_B.$$
%where $t=\max_i t_i$
Given $\phi \in B$, define
$$t_{\max}(\phi)=\max \{t_i\colon \ n_i(\phi) > 0\}.$$

\noindent {\it Claim I.} Let $\mathcal{I}$ be the family of all characteristic functions of intervals in $I$. Then $t_{\max}$ is constant on $\mathcal{I}$. Let $t_{\max}(\mathcal{I})$ be  the valued of $t_{\max}$ on $\mathcal{I}$. There is $C_1$ such that 
$$\frac{1}{C_1} |Q|^{-t_{\max(\mathcal{I})}}
  \leq ||1_Q||_B\leq C_1 |Q|^{-t_{\max(\mathcal{I})}}$$
  
  \ \\
\noindent Given a interval $Q \subset I$ there is an affine map $\psi$ such that $\psi_Q(Q)=I$, with $|\psi'|=|I|/|Q|=1/|Q|$,  so
$$1_Q = 1_I\circ \psi$$
and due the the almost homogeneity of the norm there is $K_i> 0$ such that 
$$\frac{1}{k_i} |Q|^{-t_i} n_i(1_I)
  \leq u_i(1_Q) \leq K_i |Q|^{-t_i}n_i(1_I)$$
for every interval $Q\subset I$. It follows that $t_{\max}$ is constant on $\mathcal{I}$. Note also that 
\begin{align*} ||1_Q||_B&\leq  K_{t_{\max}}^{-1}|Q|^{-t_{\max}} u_{t_{\max}}(1_I)+ \sum_{t_i < t_{\max}}  K_{t_i}^{-1}|Q|^{-t_i} u_{t_i}(1_I)\\
&\leq K_{t_{\max}}^{-1}|Q|^{-t_{\max}} \Big( u_{t_{\max}}(1_I)+ \sum_{t_i < t_{\max}}  \frac{K_{t_i}^{-1}}{K_{t_{\max}}^{-1}} |Q|^{t_{\max} -t_i} \frac{u_{t_i}(1_I)}{u_{t_{\max}}(1_I)} \Big)\\
&\leq C|Q|^{-t_{\max}} ||1_I||_B.
\end{align*} 
and the opposite inequality  is obtained with a similar argument.  This finishes the proof of the claim.

Let  $\mathcal{P}_k$ be   the Markov partition of $T^k$.  Given  an interval $P\in \mathcal{P}^{k+1}$, chose $Q_1,Q_2\in \mathcal{P}^{k+2}$ such that $Q_i\subset P$, with $i=1,2$, and $Q_1\neq Q_2$. Define
 $$a_P= \frac{1_{Q_1}}{|Q_1|}- \frac{1_{Q_2}}{|Q_2|}.$$
 Note that 
 $$\int a_P \ dm=0.$$ 

\noindent {\it Claim II.} There is $C > 0$ and $\lambda \in (0,1)$ such that for every $P\in \mathcal{P}^k$ we have
$$||a_{P}||_B\geq C\lambda^{-k}.$$
\ \\ 
\noindent Since $L$ has spectral gap on $B$ there is $C\geq 0$, $\lambda\in (0,1)$ such that for every function $a \in B$ such that    
$$\int a \ dm =0,$$
we have
%$$0<  ||a_{I}||_B=||a_{T^kP}||_B=||L^ka_{P}||_B\leq C\lambda^k ||a_{P}||_B,$$
$$ ||L^ka||_B\leq C\lambda^k ||a||_B,$$
So given $P\in \mathcal{P}^{k+1}$ we have 
$$ ||L^ka_P||_B\leq C\lambda^k ||a_P||_B,$$
On  the other hand since $B$ is continuously embedded in $L^1(m)$ there is $C> 0$ such 
$$||L^ka_P||_B\geq C ||L^ka_P||_{L^1(m)}=C||a_P||_{L^1(m)}=2C> 0,$$ 
Consequently  there is $C> 0$ such that %$$||a_{P}||_B\geq 2C\lambda^{-k}||a_I||_B\geq C2^{\beta k} ||a_I||_B $$
$$||a_{P}||_B\geq C\lambda^{-k}.$$

\noindent{\it Claim III.}  We have $t_{\max}(\mathcal{I}) >-1$. \\

It follows from  Claim II. that for every    $P\in \mathcal{P}^{k+1}$ there is $Q\in \mathcal{P}^{k+2}$, with $Q \subset P$, satisfying 
$$||\frac{1_Q}{|Q|} ||_B\geq \frac{C}{2}\lambda^{-k}. $$
Let $$\alpha = \min\frac{1}{|DT|}< 1.$$
Then
$$|Q|\geq \alpha^{k+1},$$
It follows that 
$$||1_Q ||_B\geq \frac{C}{2}\lambda^{-k}|Q|\geq \frac{C}{2}\alpha^{-k\frac{\ln \lambda}{\ln \alpha}}|Q|\geq C |Q|^{1-\beta},  $$
with
$$\beta = \frac{\ln \lambda}{\ln \alpha} >0.$$
Due Claim II. that implies $1+t_{\max}(\mathcal{I})\geq \beta>0$. \\

\noindent {\it Claim IV.} We have $t_{\max}(\mathcal{I}) \leq   0$.  

Suppose that $t_{\max}(\mathcal{I})> 0$. Let $C_1$ be as in Claim I. Choose $\theta$ such that $\theta \in (0,1/4)$,  $\theta^{t_{\max}(\mathcal{I})} \in (0,1/4)$ and such that 
$$\frac{1}{C_1}\theta^{-t_{\max}(\mathcal{I})}> 2C_1.$$
Let $Q_i=[x_i,x_{i+1}]$, with 
$x_0=0$ and 
$$|Q_i|=\theta^i$$
Then by Claim I  
\begin{align*}
C_1 |Q_0|^{-t_{\max}(\mathcal{I})}&\geq C_1 |[0,x_{k+1}]|^{-t_{\max}(\mathcal{I})}\\
&\geq 
||1_{[0,x_{k+1}]}||_B= ||\sum_{i=0}^k1_{[x_i,x_{i+1}]}||_B\\
&\geq ||1_{Q_k}||_B - \sum_{i=0}^{k-1} ||1_{Q_i}||_B\\
&\geq \frac{1}{C_1}\theta^{-kt_{\max}(\mathcal{I})}-\sum_{i=0}^{k-1} C_1 \theta^{-it_{\max}(\mathcal{I})}\\
&\geq \frac{1}{C_1}\theta^{-kt_{\max}(\mathcal{I})}-C_1 \theta^{-(k-1)t_{\max}(\mathcal{I})} \sum_{i=0}^{k-1}  \theta^{it_{\max}(\mathcal{I})}\\
&\geq \frac{1}{C_1}\theta^{-kt_{\max}(\mathcal{I})}-C_1 \theta^{-(k-1)t_{\max}(\mathcal{I})} \geq C_1 \theta^{-(k-1)t_{\max}(\mathcal{I})}.
\end{align*}
so if 
$k$ is sufficiently large, we arrive at a contradiction. This concludes the proof of Claim IV. \\

\noindent  Now we can conclude the proof of Theorem~\ref{thm2}. We know that $-1 < t_{\max}(\mathcal{I})  \leq  0$. If $t_{\max}(\mathcal{I})=0$ we can apply Theorem \ref{NEW} and obtain Case I. Otherwise we can take $s=1+t_{\max}(\mathcal{I}) \in (0,1)$. The remaining conclusions follow from Theorem \ref{main}.

 \end{proof}

%%%%%%%%%%%%%%%%%%%%%%
\section*{Acknowledgements}
%%%%%%%%%%%%%%%%%%%%%%

\begin{small}
    O.B. was partially supported by PRIN Grant ``Regular and stochastic behaviour in dynamical systems" (PRIN 2017S35EHN).
    O.B. acknowledges the MIUR Excellence Department Projects awarded to the Department of Mathematics, University of Rome Tor Vergata, CUP E83C18000100006, CUP E83C2300033000s6.
    This research is part of the O.B.'s activity within the UMI Group ``DinAmicI'' and the INdAM group GNFM.  
    D.S.  was financed  by the S\~ao Paulo Research Foundation (FAPESP), Brasil, Process Number 2017/06463-3, and Bolsa de Produtividade em Pesquisa CNPq-Brazil 311916/2023-6.
\end{small}

% %%%%%%%%%%%%%%%%%%%%%
% \section*{Statements and Declarations}
% %%%%%%%%%%%%%%%%%%%%%

% \begin{small}
%    The authors have no relevant financial or non-financial interests to disclose.
%    Data sharing is not applicable to this article as no datasets were generated or analysed during the current study.
% \end{small}

\bibliographystyle{plain}
%\nocite{*}
\bibliography{references.bib}
%\bibliography{main.bbl}

\end{document}